\documentclass[oneside, 10pt]{amsart}

\usepackage{amsmath, amsfonts, amssymb, amsthm, graphics, times}

\newtheorem{theorem}{Theorem}[section]
\newtheorem{lemma}[theorem]{Lemma}
\newtheorem{corollary}[theorem]{Corollary}
\newtheorem{proposition}[theorem]{Proposition}

\theoremstyle{definition}

\newtheorem{remark}[theorem]{Remark}

\numberwithin{equation}{section}

\begin{document}

\title{A criterion for normality of analytic mappings}

\begin{abstract}
In this  paper we give a generalization and improvement of the Pavlovi\'{c} result on the characterization of continuously differentiable functions
in the Bloch space on  the unit ball in $\mathbb{R}^m$. Then we derive a Holland--Walsh type theorem for analytic  normal mappings on the unit disk.
\end{abstract}

\author{Marijan Markovi\'{c}}

\address{
Faculty of Sciences and Mathematics\endgraf
University of Montenegro\endgraf
D\v{z}ord\v{z}a Va\v{s}ingtona bb\endgraf
81000 Podgorica\endgraf
Montenegro\endgraf
}
\email{marijanmmarkovic@gmail.com}

\keywords{normal analytic function, Bloch analytic function, hyperbolic distance, spherical distance, Holland--Walsh type characterisation, Bloch type
spaces, Lipschitz type  spaces}

\subjclass[2010]{Primary 30D45; Secondary 30H30, 32A18}

\maketitle

\section{Introduction}

In the eighties Holland and Walsh \cite{HOLLAND.WALSH.MATH.ANN}  published an interesting result on the characterisation of analytic Bloch functions on  the
unit disc which involves the expression $ |f(z) - f(w)|/|z-w|$   multiplied  with an  appropriate  weight function  depending on the both  variables $z$ and
$w$. More precisely,         they obtained  that an analytic function $f(z)$ belongs to the  Bloch space  on the unit disc  $\mathbb{U}$  if and only if the
expression
\begin{equation*}
\sqrt {1-|z|^2}  \sqrt{1-|w|^2} \frac {|f(z) - f(w)|}{|z-w|}
\end{equation*}
is bounded  for $z,\, w\in \mathbb{U}$, $z\ne  w$.

More  recently the  same  characterisation was given for  analytic Bloch functions in the unit ball of $\mathbb{C}^m$ by Ren and Tu \cite{REN.PAMS}.  After
that  Ren and K\"{a}hler \cite{REN.PEMS}  proved  this characterisation for harmonic functions on  the unit ball in $\mathbb{R}^m$.

In 2008,        Pavlovi\'{c} \cite{PAVLOVIC.PEMS}  proved that even  continuously differentiable Bloch functions  obey the same  characterisation. Actually,
Pavlovi\'{c}  proved more  in the following  proposition.

\begin{proposition}[Cf. \cite{PAVLOVIC.PEMS}]\label{PR.PAVLOVIC}
A continuously differentiable  complex-valued  function $f$ on  the unit ball in  $\mathbb{R}^m$  is a Bloch  function, i.e.,
\begin{equation*}
\sup_{|x|<1}  (1-|x|^2) \|D _ f (x)\|
\end{equation*}
is finite, if and only if the following quantity if finite
\begin{equation*}
\sup_{|x|,\,  |y|<1,\, x\ne y }\sqrt { 1-|x|^2 }\sqrt{  1-|y |^2  }  \frac{|f(x) - f (y)|}{| x - y|}.
\end{equation*}
Moreover, the above  numbers are equal.
\end{proposition}

Therefore, by this   proposition the Bloch semi--norm
\begin{equation*}
\|f\|  =  \sup_{|x|<1}  (1-|x|^2) \|D _ f (x)\|.
\end{equation*}
of a continuously differentiable (complex--valued) function $f$ on  the unit ball in $\mathbb{R}^m$  may be expressed in the   differential--free way
\begin{equation*}
\|f\| = \sup_{|x|,\, |y|<1,\, x\ne y}  \sqrt { 1-|x|^2 }\sqrt{  1-|y |^2}  \frac{|f(x)-f (y)|}{|x-y|}.
\end{equation*}

For an analytic  function $f$ on $\mathbb{U}$ we  have  $|f'(z)| = \|D_f(z)\|$ (for the right side we consider  $D_f(z)$  as a linear   mapping   from
$\mathbb{R}^2$ into  $\mathbb  {R}^2$), so from  the Pavlovi\'{c} result we can recover the  characterization of analytic functions in the Bloch space
on the unit  disc obtained by Holland  and Walsh \cite[Theorem 3]{HOLLAND.WALSH.MATH.ANN} as well as the Ren and Tu results.

Let  $\mathrm {Aut} (\mathbb{U})$  be the group of all conformal  transforms of the unit disk onto itself. An  analytic  function $f$ on the unit disk
is  Bloch if and only if  $\{(f\circ\varphi)(z) - (f\circ\varphi)(0):\varphi\in \mathrm {Aut} (\mathbb{U})\}$ is a normal family \cite{COLONNA.PALERMO}.

One says that  an analytic function $f$ is normal on $\mathbb{U}$ if $\{(f\circ\varphi)(z):\varphi\in\mathrm {Aut}(\mathbb {U})\}$  is a normal family.
It is well known that an analytic function $f$ on the unit disc is normal  if  and only  if  it  satisfies  the growth  condition
\begin{equation*}
\frac {|f'(z)|}{1+|f(z)|^2}\le \frac C {1-|z|^2},\quad z\in \mathbb{U}
\end{equation*}
for a constant $C>0$.

The main aim of this article is to obtain   a new criterion for normality of analytic functions on $\mathbb{U}$. This  criterion is stated in the
proposition which  follows. A proof of this proposition follows from the characterisation result  (given in  the main  lemma in the next section)
similar to Proposition \ref{PR.PAVLOVIC}  for continuously differentiable mappings that satisfy a certain  growth condition.

\begin{proposition}
Let $f$ be an analytic function on the disk $\mathbb{U}$.  The function $f$ is  Bloch  if and only  if
\begin{equation*}
{\sqrt{1-|z|^2} \sqrt{1-|w|^2}} \frac {|f(z) -  f(w)|}{|z-w|}
\end{equation*}
is bounded as a function of  $z,\, w\in \mathbb{U} $  for $z\ne w$.

The function $f$  is normal if and only if
\begin{equation*}
 \frac{\sqrt{1-|z|^2} \sqrt{1-|w|^2}}{\sqrt{1+|f(z)|^2} \sqrt{1+|f(w)|^2}} \frac {|f(z) -  f(w)|}{|z - w|}
\end{equation*}
is  bounded as a function of     $z,\, w\in \mathbb{U}  $  for $z\ne w$.
\end{proposition}

\section{The main lemma}

We will  introduce here  the needed notation and terminology.

Let  $\mathbb{B}^m$ be  the  unit  ball in $\mathbb{R}^m$.

For a differentiable mapping $f:D\to\mathbb{R}^n$, where $D\subseteq\mathbb{R}^m$ is a domain, we denote by $D_f(x)$ its differential at $x \in D$, and by
\begin{equation*}
\|D_f (x)\|  = \sup _{\zeta\in \partial \mathbb{B}^{m}} |D _ f (x)\zeta|
\end{equation*}
the norm of the linear operator $D_f(x):\mathbb{R}^m\to\mathbb{R}^n$.
The class of all continuously differentiable mappings $f:D\to\tilde{D}$  is denoted by $\mathcal{C}^1 (D,\tilde{D})$.

A weight function is an everywhere positive and continuous function on a domain in $\mathbb{R}^m$.        If  $\omega$ is a weight function on a domain
$D\subseteq  \mathbb{R}^m$, the  $\omega$-distance between $x\in D$ and $y\in D$ is given by
\begin{equation*}
d_\omega  (x,y)  =  \inf_\gamma \int_\gamma {\omega(z)}  {|dz|},
\end{equation*}
where $\gamma\subseteq D$   is  among  all piecewise $\mathcal{C}^1$-curves connecting $x$  and $y$.

Let $\omega$   and  $\tilde{\omega}$  be weight functions on domains  $D\subseteq \mathbb{R}^m$ and $\tilde{D}\subseteq \mathbb{R}^n$, respectively.  We
will consider mappings  $f\in \mathcal {C}^1 (D,\tilde{D})$  which  satisfy  the   Bloch type  growth  condition, i.e., the   growth condition  of   the
type
\begin{equation*}
 {\tilde{\omega}  (f( x ))}  \|D_ f(x)\|\le C{\omega(x)},  \quad  x\in  D,
\end{equation*}
where $C$ is a positive   constant. For  such   mappings we  introduce
\begin{equation*}
\mathbf {B}_f  =  \sup_{x\in D} \frac  {\tilde{\omega} (f(x)) }  {\omega(x)}\|D_ f(x)\|,
\end{equation*}
which will be called the Bloch number  of the mapping $f$. We denote by $\mathcal {B}_{\omega,\tilde{\omega}}$ the class  of  all mappings
$f\in \mathcal {C}^1(D,\tilde{D})$ for which the Bloch  number  $\mathbf{B}_f$ is finite.

Note that for $\tilde{D} = \mathbb{R}^n$ and $\tilde{\omega}\equiv 1$  the Bloch number $\mathbf{B}_f$ has the semi--norm properties. Moreover, the class
$\mathcal {B}_{\omega,\tilde{\omega}}$  has  the  linear   space  structure.

The main  aim of  this section   is to  obtain a  differential--free  description of the class  $\mathcal {B}_{\omega,\tilde{\omega}}$ and  the
differential--free expression  for  the    Bloch  number of a continuously  differentiable  mapping. In order to do that,  we will  consider mappings
$f\in \mathcal {C}^1 (D,\tilde{D})$  which    satisfy the  Lipshitz  type growth condition, i.e.,
\begin{equation*}
|f(x) - f(y) | \le C_f(x,y)    |x -  y|,
\end{equation*}
where $C_f(x,y)$ is a  positive function.

For given weight functions  $\omega$  on  $D$  and $\tilde{\omega}$ on $\tilde{D}$, and a mapping $f\in \mathcal {C}^1(D,\tilde{D})$ we  introduce an
everywhere  positive function $\Omega_f $ on  $D\times  D$  such that  the  following  conditions are satisfied:
\begin{equation*}
{\Omega}_f  (x,y)  =  {\Omega}_f (y,x),\quad
{\Omega}_f(x,x) = \frac{\tilde{\omega} (f(x))} {\omega(x)},\quad
\liminf_{z\to x}  {\Omega}_f(x,z)\ge  {\Omega}(x,x),
\end{equation*}
and
\begin{equation*}
{\Omega }_f(x, y ) \frac{ |f(x) -  f(y) | }{ |x - y | } \le \frac {d_{\tilde{\omega}}(f(x),f(y))}{d_\omega(x,y)}, \quad x,\, y\in D,\,   x\ne y.
\end{equation*}
We  say that ${\Omega}_f$ is an admissible  function   for the mapping  $f$ with respect to  the given weight  functions $\omega$  and $\tilde{\omega}$.

Note that if ${\Omega}_f $ is not symmetric,  but satisfies all other  conditions   stated above, we  can replace   it  by  the  symmetric function
\begin{equation*}
\tilde{\Omega}_f(x,y)= \max\{{\Omega}_f(x,y), \Omega_f(y,x)\},\quad x,\, y\in D.
\end{equation*}
This new function will be   admissible for the same mapping,   as it is easy to  check.

Also  note that if we take  $\tilde{\omega}\equiv 1$  on  the domain $\tilde{\Omega}$, then the distance $d_{\tilde{\omega}}$ is equal to the Euclidean
distance. In this case the fourth condition of admissibility  is  independent  of  the  mapping  $f$, and  it reduces on finding an  universal
admissible function ${\Omega} (x, y ) $ which satisfies the simplified condition
\begin{equation*}
{\Omega}(x,y) {d_\omega(x,y)}\le |x-y|,\quad x,\, y\in D.
\end{equation*}

Of course, the admissible function need not be unique, and one may pose the existence question. In the remark given  below we solve the existence question
in the general  setting.

Introduce  now  the  following    quantity
\begin{equation*}
\mathbf{L}_f   = \sup _{x,\, y\in D,\, x\ne y}  {\Omega }_f(x,y)\frac {  | f(x)  -  f(y)| }{  |x - y | },
\end{equation*}
where $\Omega_f $ is an admissible function for the mapping $f$ with respect to $\omega $ and $\tilde{\omega}$.  We call it the Lipschitz number of $f$.
The main lemma  stated below says  that the Lipschitz number  does not depend on the choice of an  admissible  function  (therefore   the definition of
$\mathbf{L}_f$  is correct).

The class of all mappings  $f\in \mathcal {C}^1 (D,\tilde{D})$ for which the  Lipschitz number $\mathbf{L}_f$ is  finite is denoted by
$\mathcal {L}_{\omega,\tilde{\omega}}$. If  $\tilde{D}  = \mathbb{R}^n$  and $\tilde{\omega}=1$  then
$\mathbf{L}_f$ also   has the semi--norm properties, and  $\mathcal {L}_{\omega,\tilde{\omega}}$ is a  linear   space.

Now  we prove our  main Lemma \ref{LE.MAIN} which connects the Bloch and Lipshitz number of a continuously differentiable  mapping between Euclidean
domains. Our main  lemma  shows that  any   mapping $f\in \mathcal {C}^{1} (D, \tilde{D})$ satisfies
\begin{equation*}
\mathbf{B}_f =\mathbf{L}_f.
\end{equation*}

As a consequence we have  that the  Lipschitz number is independent of the choice of the admissible function $\Omega_f$, and the  Bloch  number   may   be
expressed in the differential--free way
\begin{equation*}
\mathbf{B} _ f = \sup_{x,\, y\in D,\, x\ne y }   {\Omega }_f(x,y)    \frac {  | f(x)  -  f(y)| }{  |x - y | },
\end{equation*}
where $\Omega_f$ is an  admissible function  for $f$.

As another consequence we have the  coincidence of the two classes  of  mappings in  $\mathcal {C}^1(D,\tilde{D})$, i.e.,
$\mathcal {B}_ {\omega,\tilde{\omega}} = \mathcal {L} _{\omega,\tilde{\omega}} $. Thus, the Bloch  class $\mathcal {B}_ {\omega,\tilde{\omega}}$
may be described as
\begin{equation*}
\left \{ f \in C^1 (D,\tilde{D}): \sup _{x,\, y\in D,\, x\ne y}{\Omega }_f (x,y) \frac{ | f(x)- f(y) | }{|x - y |}<\infty\right\}.
\end{equation*}

All the results and facts stated above follow from   the content of the  following lemma.

\begin{lemma}\label{LE.MAIN}
Let $(D,d_\omega)$ and $(\tilde{D},d_{\tilde{\omega}})$ be domains in $\mathbb{R}^m$ and $\mathbb{R}^n$ with distances  $d_\omega$ and $d_{\tilde{\omega}}$
generated by the weight functions $\omega$ and $\tilde{\omega}$  in $D$ and $\tilde{D}$, respectively. Let  $f\in C^1 (D,\tilde{D})$,  and let ${\Omega}_f$
be any admissible function for the mapping $f$ with  respect to $\omega$ and $\tilde{\omega}$.
If one of the numbers   $\mathbf{B}_f$   and  $\mathbf{L}_f$  is  finite, then  both  numbers are finite, and these numbers are equal.
\end{lemma}

\begin{proof}
For one direction,  assume that the Lipschitz number  of  $f\in \mathcal {C}^1(D,\tilde{D}) $, i.e., that the  quantity
\begin{equation*}
\mathbf  {L}_f  =  \sup _{ x\ne y }   {\Omega }_f(x,y)  \frac {  | f(x)  -  f(y)| }{  |x - y | }
\end{equation*}
is finite, where $\Omega_f$ is an admissible function  for $f$ with respect to $\omega$ and $\tilde{\omega}$. We  are going to show that
$\mathbf  {B}_f \le \mathbf {L}_f$,     which  implies that the Bloch number  $\mathbf {B}_f$ must   also be  finite.

Let $x \in \Omega$.  If we have  in mind that
\begin{equation*}
\limsup_{y\to x} \frac{|f(x) - f(y)|}{|x - y |} = \|D _ f(x)\|,
\end{equation*}
we obtain
\begin{equation*}\begin{split}
\mathbf{L}_f&  =  \sup _{  y\ne  z }   {\Omega }_f  (y,z) \frac { |f(y - f(z) |}{| y - z|}
\ge \limsup_{z \rightarrow x} {\Omega}_f(x,z)\frac {  | f(x) -  f(z) |}{ |x - z |}
\\& \ge \liminf_{z\rightarrow x} {\Omega}_f(x,z) \limsup_{z\rightarrow x}\frac {  |f(x)  -  f(z) |}{|x - z |}
\ge {\Omega}_f (x,x) \|D_f(x) \| \\&= \frac {\tilde { \omega}(f(x))}{\omega(x)} \|D_f(x) \|.
\end{split}\end{equation*}
We have used the   fact that
\begin{equation*}
\limsup _{y\rightarrow  x } A (y) B (y)   \ge  \liminf_{y\rightarrow x}   A (y)  \limsup_{y\rightarrow  x } B(y)
\end{equation*}
for non--negative   functions  $A $  and       $B $  on  an Euclidean domain.

It follows that
\begin{equation*}
\mathbf{L}_f     \ge\sup_{x\in  D }  \frac {\tilde { \omega}(f(x))}{\omega(x)}  \| D_f(x) \| = \mathbf{B}_f,
\end{equation*}
which we aimed to prove.

Assume now that the  Bloch number    $\mathbf{B}_f$ of a continuously differentiable mapping  $f: D \to\tilde{D}$  is finite. We  will prove the reverse
inequality $\mathbf{L}_f\le\mathbf {B}_f$,  which in particular implies that the  Lipschitz  number  $\mathbf {L}_f$  is also  finite.

Let $\gamma\subseteq D$ be  any piecewise  $\mathcal {C}^1$-curve connecting $x\in D$ and $y\in D$, i.e., such that $\gamma (0)= x$  and $\gamma(1) = y$.
Since  $f\in \mathcal {C}^1(D,\tilde {D})$, the curve  $\delta  =    f\circ\gamma \subseteq \tilde{\Omega}$ (which connects  $f(x)$ and $f(y)$), is also
piecewise  $\mathcal {C}^1$ in the domain $\tilde{D}$  and we have
\begin{equation*}\begin{split}
d_ {\tilde {\omega  }} (  f(x)  ,  f(y) )  & \le
\int_0^1 \tilde{\omega}  (\delta (t)) |\delta'(t) |dt
=  \int_0^1 \tilde{\omega} (f\circ \gamma (t))  | D _f ( \gamma (t) ) \gamma'(t)| dt
\\& \le\int_0^1\tilde{\omega} (f\circ  \gamma (t))   \| D_f ( \gamma (t) )\|   |\gamma'(t) | dt
 \le B _ f   \int_0^1  \omega  (\gamma (t))  |\gamma'(t)|dt
\\&= \mathbf {B}_f d_\omega (x,y).
\end{split}\end{equation*}
If   we now  take the  infimum over all curves  $\gamma$ we obtain
\begin{equation*}
\frac {d_{\tilde{\omega}} (f(x),  f(y))}{  d_{\omega} (x,y)}\le \mathbf{B}_f
\end{equation*}
for every  $x\in D$  and $y\in D$ such that $x\ne y$.  Applying now   conditions    posed on   the  admissible function ${\Omega}_f$,   we obtain
\begin{equation*}
{\Omega}_f (x,y)  \frac {|  f(x) -   f(y)|}{|x -y|} \le  \frac {d_{\tilde{\omega}} (f(x),  f(y))}{ d_{\omega} (x,y)}\le \mathbf{B}_f.
\end{equation*}
It follows that
\begin{equation*}
\mathbf{L}_f  = \sup_{x,\, y\in D,\, x\ne y}  {\Omega}_f(x,y)    \frac { |f(x)  -  f(y)| }{ |x-y|} \le  \mathbf{B}_f,
\end{equation*}
which we aimed  to prove.
\end{proof}

\begin{remark}\label{RE.OMEGA}
Let us first note that if $D\subseteq \mathbb{R}^m$ is  a  domain, and $\omega$ a weight function on $D$, then for the  $\omega$-distance $d_\omega$ on
$\Omega$, we have
\begin{equation*}
\lim _{y\to x}\frac { d_\omega(x,y)} {|x-y|}  = \omega (x), \quad x\in  \Omega.
\end{equation*}

Indeed, since  $\omega $ is continuous, there exists an open ball $B(x,r)\subseteq \Omega$ such that
\begin{equation*}
0<\omega(x)-\varepsilon<\omega (y)<\omega(x)+\varepsilon,\quad y\in B (x,r),
\end{equation*}
where $\varepsilon> 0$  is  a  sufficiently small   number. Now,  we have
\begin{equation*}
d_\omega (x,y)  \le \int_{[x,y]}  \omega  \le   (\omega (x) +\varepsilon )  |x -  y | =  (\omega (x) +\varepsilon ) |x-y|.
\end{equation*}
On the other hand, if $\gamma\subseteq \Omega$ is among curves  that connect   $x$  and $y$, then
\begin{equation*}
d_{\omega} (x,y) = \inf_\gamma \int_\gamma \omega\ge (\omega (x) - \varepsilon)|x-y|.
\end{equation*}
Therefore,
\begin{equation*}
\omega(x) - \varepsilon  \le \frac{d_\omega (x,y)}{|x-y|} \le \omega(x) + \varepsilon,\quad    y\in B (x,r).
\end{equation*}
This means that
\begin{equation*}
\lim _{y\to  x}\frac { d_\omega(x,y)} {d (x,y)} = \omega (x).
\end{equation*}

Let us now solve the existence question concerning the   admissible function. Let $f\in \mathcal {C}^1(D,\tilde{D})$ satisfies the condition
$\tilde{\omega}(f(x))  \|D_f (x)\|\le C \omega (x)$, $x\in D$. Then
\begin{equation*}
\Omega _ f(x,y) =
\begin{cases}
\frac{ d_{\tilde{\omega}}(f(x),f(y))}{|f(x)- f(y) |}/
\frac {d_\omega(x,y)}{|x -y |}, & \mbox{if}\ x\ne y,f(x)\ne f(y); \\
 \tilde{\omega} (f(x)) /   \frac {d_\omega(x,y)}{| x -y |}, & \mbox{if}\ x\ne y, f(x)=f(y) ;\\
{\tilde{\omega} (f(x))} /{\omega(x)},  & \mbox{if}\ x=y,f(x)=f(y).
\end{cases}
\end{equation*}
is an admissible function for $f$. Having in mind the preceding remark  it follows
\begin{equation*}
\liminf_{y\to x} \Omega_f (x,y) = \lim_{y\to x} \Omega_f (x,y) =\frac{\tilde{\omega} (f(x))} {\omega(x)} =  \Omega _f(x,x).
\end{equation*}
Other           three admissability    conditions for $\Omega_f$ are obviously satisfied.
\end{remark}

In view of Remark \ref{RE.OMEGA}  we  have the following expected corollary.

\begin{corollary}
Let $(D,d_\omega)$ and $(\tilde{D},d_{\tilde{\omega}})$ be domains  in $\mathbb{R}^m$ and $\mathbb{R}^n$ with distances $d_\omega$ and  $d_{\tilde{\omega}}$
generated by the weight functions $\omega$ and $\tilde{\omega}$  in $D$ and  $\tilde{D}$, respectively.
Then $f$ satisfies the  inequality
\begin{equation*}
\tilde {\omega} (f(x)) \|D_f(\zeta)\| \le C \omega (x),
\end{equation*}
where $C$ is  a positive  constant,  if and only if there holds
\begin{equation*}
d_{\tilde {\omega} } (f(x),f(y))\le C  d_\omega (x,y)
\end{equation*}
for the same constant $C$.
\end{corollary}

For example, the result of the last corollary is  proved for  harmonic mappings of the unit disc into itself by  Colonna in \cite{COLONNA.IUMJ}, where it is
also found  that the constant $C$ is  less or equal to $ 4/\pi$  for such type of mappings. A variant of this  corollary is obtain in \cite{ZHU.JLMS}   (see
also Theorem 1 there for analytic  functions of several complex variables). A variant is also given in \cite{MARKOVIC.CMB}.

\section{Characterisations of Bloch and normal mappings}

Based  on our main lemma  one may  derive Proposition \ref{PR.PAVLOVIC}. Indeed, if we take $D = \mathbb{B}^m$,  $\omega (x) =  1/( 1-|x|^2) $,
$x\in \mathbb{B}^m$, then $d_\omega$ is the  hyperbolic distance on  the  unit ball which will be denoted by $d_h$.  It is  well known that
\begin{equation*}
d_h(x,y)  =  \mathrm {asinh}\, \frac{|x - y |}{\sqrt {1-|x|^2}\sqrt {1-|y^2}} \quad x,\, y\in \mathbb{B}^m.
\end{equation*}
On the other hand, take $\tilde{D} = \mathbb {R}^n$ and $\tilde {\omega} \equiv 1$. Then $d_{\tilde {\omega}}$ is the Euclidean  distance.

The function  ${\Omega}(x,y) = \sqrt {1-|x|^2} \sqrt {1-|y|^2}$  satisfies  the  inequality
\begin{equation*}
d_h (x,y)   {\Omega}(x,y)\le |x - y|,\quad x,\,  y\in\mathbb{B}^m,
\end{equation*}
and therefore it is admissible   for any  $f \in \mathcal {C}^1(\mathbb{B}^m, \mathbb{R}^n)$ with the growth estimate
\begin{equation*}
(1-|x|^2) \|D_f (x)\| \le  C,\quad x\in \mathbb{B}^m
\end{equation*}
for a  constant $C$.

Indeed, using the inequality $\mathrm {asinh}\, t\le t$ for $t\ge 0$ (to prove it, let $\phi(t) = \mathrm {asinh}\,  t  -  t $; then we have $\phi (0) = 0$
and  $\phi'(t) = \frac{1}{\sqrt{1+t^2} }-1<0$, $t>0$, so the inequality  follows  from $\phi(t)\le \phi (0) = 0$)   one deduces:
\begin{equation*}\begin{split}
\frac { |x - y|}{d_h(x,y) }
& = {|x-y |} : { \mathrm {asinh}\, \frac{|x - y  |}{\sqrt {1-|x|^2}\sqrt {1-|y|^2}}}
\\&\ge {|x-y |} : {   \frac{|x - y  |}{\sqrt {1-|x|^2}\sqrt {1-|y|^2}}}
\\&=  \sqrt {1-|x|^2} \sqrt {1-|y|^2}  = {\Omega}(x,y),\quad x,\, y\in {\mathbb{B}^m},\, x\ne y.
\end{split}\end{equation*}
The Pavlovi\'{c} result in this case  now follows.  We  gave    a  similar  proof  in  \cite{MARKOVIC.CMB}.

Applying the following theorem  for normal analytic function on $\mathbb{U}$  we immediately obtain  the proposition stated  in the Introduction.

\begin{theorem}
A continuously   differentiable  mapping $f:\mathbb{B}^m\to \mathbb{R}^n$ satisfies the growth  condition
\begin{equation*}
\frac1 { 1-|x|^2} \|D _ f(x)\|\le  \frac {C}{ 1+|f(x)|^2},\quad   x\in  \mathbb{B}^m
\end{equation*}
for a constant $C$, if and only if there holds
\begin{equation*}
|f(x) -  f(y)| \le C  |x-y|  \frac{\sqrt{1+|f(x)| ^2} \sqrt{1+|f(y)|^2}}{\sqrt{1-|x|^2} \sqrt{1-|y|^2}}\quad x,\, y \in \mathbb{B}^m.
\end{equation*}
\end{theorem}

\begin{proof}
In our main lemma let us take  for the domain  $D$ the unit   ball $\mathbb{B}^m$, and for the domain  $\tilde{D}$ the space $\mathbb{R}^n$.   Let moreover
$\omega (x) =  1/( 1-|x|^2)$, $x\in \mathbb{B}^m$  and $\tilde{\omega}(y)=1/(1+|y|^2)$, $y\in \mathbb{R}^n$. As we have already said, $d_{\omega}$   is the
hyperbolic  distance $d_h$ on  the unit ball $\mathbb{B}^m$. The distance $d_{\tilde{\omega}}$ is the spherical distance      on  $\mathbb{R}^n$ which will
be denoted by $d_s$.  For  the spherical   distance  we have
\begin{equation*}
d_s(x,y)  =    \frac{|x - y |}{\sqrt {1+|x|^2}\sqrt {1+|y^2}},\quad x,\, y\in \mathbb{R}^m.
\end{equation*}

Now, we will  prove that
\begin{equation*}
\Omega_f (x,y) = \frac{\sqrt{1-|x|^2} \sqrt{1-|y|^2}} {\sqrt{1+|f(x)|^2} \sqrt{1+|f(y)|^2}},\quad x,\, y\in {\mathbb{B}^m}
\end{equation*}
is an admissible  function for $f$  with  respect to  the hyperbolic and spherical  weights.

First  note that
\begin{equation*}
\Omega_f(x,x)  = \frac{ 1-|x|^2 }{ 1+|f(x)|^2 } =  \frac{ 1 }{ 1+|f(x)|^2 } : \frac{ 1 }{ 1-|x|^2 }, \quad  x\in \mathbb{B}^m.
\end{equation*}

Since $\Omega_f(x,y)$ is symmetric and continuous, it remains only   to prove that $\Omega_f(x,y)$ satisfies  the inequality
\begin{equation*}
{\Omega}_f(x,y) \frac{|f(x)-f(y)|}{|x-y|} \le \frac {d_s (f(x),f(y))}{d_h (x,y)},\quad    x,\, y\in {\mathbb{B}^m},\, x\ne y.
\end{equation*}
Having in mind  the inequality $\mathrm {asinh}\le t$,  $t\ge  0$,  we obtain
\begin{equation*}\begin{split}
\frac{d_s(f(x),f(y) )}{d_h(x,y)}
& =  { \frac{|f(x)-f(y)|}{\sqrt{1+|f(x)|^2} \sqrt{1+|f(y)|^2}}}
: {\mathrm{asinh}  \frac{|x-y|}{\sqrt{1-|x|^2} \sqrt{1-|y|^2}}}
\\&\ge { \frac{|f(x)-f(y)|}{\sqrt{1+|f(x)|^2} \sqrt{1+|f(y)|^2}}}
 : {\frac{|x-y|}{\sqrt{1-|x|^2} \sqrt{1-|y|^2}}}
\\&= \frac{\sqrt{1-|x|^2} \sqrt{1-|y|^2}}{\sqrt{1+|f(x)|^2} \sqrt{1+|f(y)|^2}}
 \frac{|f(x)-f(y)|}{|x-y|} \\&= \Omega_f(x,y) \frac{|f(x)-f(y)|}{|x-y|},
\end{split}\end{equation*}
which   we aimed to prove.
\end{proof}

\end{document}